\documentclass[11pt]{amsart}

\usepackage{amsmath,amssymb,amscd,amsfonts,graphicx,latexsym,epsfig,psfrag,url,color,xr,mathtools,enumitem,transparent,fullpage,hyperref,float,mathrsfs,wrapfig,rotating,array}
\usepackage[all]{xy}
\usepackage[foot]{amsaddr}

\newtheorem{theorem}{Theorem}[section]

\newtheorem{lemma}[theorem]{Lemma}

\newtheorem{defprop}[theorem]{Definition-Proposition}

\theoremstyle{definition}
\newtheorem{definition}[theorem]{Definition}
\theoremstyle{remark}

\theoremstyle{plain}
\newcommand{\thistheoremname}{}
\newtheorem{genericthm}[theorem]{\thistheoremname}

\newtheorem*{genericthm*}{\thistheoremname}
\newenvironment{namedthm*}[1]
  {\renewcommand{\thistheoremname}{#1}%
   \begin{genericthm*}}
  {\end{genericthm*}}
  
\newcommand{\bR}{\mathbb{R}}

\newcommand{\bZ}{\mathbb{Z}}

\newcommand\ba{\mathbf{a}}

\newcommand\bm{\mathbf{m}}
\newcommand\bn{\mathbf{n}}

\newcommand\bx{\mathbf{x}}
\newcommand\boldQ{\mathbf{Q}}
\newcommand\bzero{\mathbf{0}}

\newcommand{\btB}{{\mathbf{2B}}}
\newcommand{\sB}{\mathscr{B}}
\newcommand{\stB}{2\mathscr{B}}

\newcommand{\on}{\operatorname}

\newcommand{\comp}{C^2}

\renewcommand{\comp}{{\on{comp}}}

\newcommand{\mk}{{\on{mark}}}
\newcommand{\incom}{\on{in}}

\newcommand{\inte}{{\on{int}}}

\newcommand{\tree}{{\on{tree}}}
\newcommand{\br}{{\on{br}}}

\renewcommand{\min}{{\on{min}}}
\renewcommand{\max}{{\on{max}}}

\makeatletter
\newcommand*\bigcdot{{\mathpalette\bigcdot@{.5}}}
\newcommand*\bigcdot@[2]{{\mathbin{\vcenter{\hbox{\scalebox{#2}{$\m@th#1\bullet$}}}}}}
\makeatother

\newcommand\qu{/\kern-.7ex/} 
\newcommand\lqu{\backslash \kern-.7ex \backslash}

\newcommand{\sr}{\stackrel}
\newcommand{\wh}{\widehat}
\newcommand{\wt}{\widetilde}

\begin{document}

\title{The 2-associahedra are Eulerian}
\author{Nathaniel Bottman}
\address{Department of Mathematics, University of Southern California,
3620 S Vermont Ave, Los Angeles, CA 90089}
\email{\href{mailto:bottman@usc.edu}{bottman@usc.edu}}
\author{Dylan Mavrides}
\address{Department of Mathematics, Princeton University, Fine Hall, Washington Rd, Princeton, NJ 08544}
\email{\href{mailto:mavrides@princeton.edu}{mavrides@princeton.edu}}

\maketitle

\begin{abstract} 
We show that the 2-associahedra are Eulerian, by exploiting their recursive structure.
\end{abstract}

\section{Introduction}

In \cite{b:2ass}, the first author introduced the \emph{2-associahedra}, a family $(W_\bn)$ of abstract polytopes which forms the algebraic underpinnings for the naturality structure on the Fukaya category \cite{b:realization,bc:Ainfty,bw:compactness,mww}.
While these posets originate in symplectic geometry, they are natural combinatorial objects, which can be thought of either as the posets of degenerations in the configuration spaces of marked vertical lines in $\bR^2$, or as the posets that control the coherences in an $(A_\infty,2)$-category.

A ranked poset is \emph{Eulerian} if every nontrivial interval has the same number of even-rank elements as it does odd-rank elements.
The face lattice of a convex polytope is an example, and many results about the combinatorics of convex polytopes can be extended to Eulerian posets.
Our main result is to show that the 2-associahedra are Eulerian:
\begin{theorem}
\label{thm:eulerian}
For every $r \geq 1$ and $\bn \in \bZ_{\geq0}^r\setminus\{\bzero\}$, the 2-associahedron $\wh{W_\bn} \coloneqq W_\bn \cup \{F_\bn^\min\}$ with a formal dimension-$(-1)$ minimal element is an Eulerian poset.
\end{theorem}
\noindent
One motivation for this theorem is one can associate to an Eulerian poset an important invariant called the \emph{$cd$-index}, which is a concise encoding of the flag $f$-vector \cite{stanley:cd}.
By using the recursive structure of the 2-associahedra as described in \S\ref{s:equations}, we hope to compute the $cd$-indices of the 2-associahedra algorithmically in future work.
Another motivation for Thm.~\ref{thm:eulerian} is that it provides further evidence that the 2-associahedra can be realized as the face lattices of convex polytopes.

We now describe the plan for our paper.
\begin{enumerate}
\item[\bf\S\ref{s:equations}:]
In this section, we define a collection of generating functions $F_T$, where $F_T$ counts the elements of the 2-associahedra that project to a given element $T$ of an associahedron.
The main result of \S\ref{s:equations} is Thm.~\ref{thm:gen_functions}, which reformulates the recursive structure of the 2-associahedra as a set of equations satisfied by the collection $\bigl(F_T\bigr)$.

\medskip

\item[\bf\S\ref{s:eulerian}:]
We exploit Thm.~\ref{thm:gen_functions} to prove Thm.~\ref{thm:eulerian}.
Specifically, we divide the proof of Thm.~\ref{thm:eulerian} among \S\S\ref{ss:Wn-hat}--\ref{ss:remaining_intervals} like so:
\medskip
\begin{enumerate}
\item[\bf\S\ref{ss:Wn-hat}:]
In this subsection, we use Thm.~\ref{thm:gen_functions} to show that the alternating sum over the completed 2-associahedron $\wh{W_\bn}$ is zero (Lemma~\ref{lem:Wn-hat_balanced}).
The key is to note that since the variable $t$ in $F_T$ tracks the dimension, we can compute the alternating sum by specializing to $t=-1$.

\medskip

\item[\bf\S\ref{ss:fib_prods}:]
Next, we use Lemma~\ref{lem:Wn-hat_balanced} to show that the completion of any fiber product $W_{\bm_1} \times_{K_r} \cdots \times_{K_r} W_{\bm_k}$ is balanced (Lemma~\ref{lem:A_fib_prod_of_Wms}), where we are using the forgetful map $W_{\bm_i} \to K_r$ defined in \cite[Thm.~4.1]{b:2ass}.
This, plus the recursive structure of 2-associahedra (see \eqref{eq:fiber_product_decomp}), allows us to show that the alternating sum over any interval in $\wh{W_\bn}$ of the form $[F_\bn^\min,2T]$ is zero (Lemma~\ref{lem:F_min_2T_balanced}).

\medskip

\item[\bf\S\ref{ss:remaining_intervals}:]
Finally, we show in Lemma~\ref{lem:remaining_intervals} that any interval in $W_\bn$ has alternating sum zero.
Together with Lemma~\ref{lem:F_min_2T_balanced}, this completes the proof of Thm.~\ref{thm:eulerian}.
The proof of Lemma~\ref{lem:F_min_2T_balanced} is somewhat involved, but is conceptually simpler than that of Lemma~\ref{lem:F_min_2T_balanced} because an interval in $W_\bn$ (as opposed to an interval in $\wh{W_\bn}$ with lower bound $F_\bn^\min$) essentially decomposes into a product, with one term for each bracket.
\end{enumerate}
\end{enumerate}

In this paper, we will assume all the definition, notation, and results from \cite{b:2ass}.

\subsection{Acknowledgments}

The first author thanks Richard Ehrenborg, Margaret Readdy, and Lauren Williams for useful conversations.
The first author was supported by an NSF Mathematical Sciences Postdoctoral Research Fellowship and by an NSF Standard Grant (DMS-1906220).

\section{Generating functions associated to 2-associahedra}\label{s:equations}

In this section, we will study the following generating functions:
\begin{gather}
f(t,x)
\coloneqq
\sum_{m \geq 0, r \geq 1} a_{mr}t^mx^r,
\quad
a_{mr}
\coloneqq
\#\left\{T \in K_r \:|\: d(T) = m\right\} \eqqcolon K_{m,r},
\\
F_T(t,\bx)
\coloneqq
\sum_{{m \geq 0,}\atop{\bn \in \bZ_{\geq0}^r\setminus\{\bzero\}}} A_{T,m,\bn}t^m\bx^\bn, \quad A_{T,m,\bn}
\coloneqq
\#\left\{2T \in W_\bn \:\left|\: {{\pi(2T) = T,}\atop{d(2T) = m}}\right.\right\}
\eqqcolon
\!W_{T,m,\bn},
\end{gather}
where $T \in K_r$ is a stable rooted ribbon tree (RRT).
That is, $f$ counts faces of the associahedra, and $F_T$ counts faces of the 2-associahedra whose image under the forgetful map is $T$.
For instance, $K_4$ is CW-isomorphic to a pentagon, so $f$ should include the terms  $5x^4$, $5tx^4$, and $t^2x^4$.

Our main result in this section is the following characterization of $f$ and $F_T$, which will be the key input to our proof that $W_\bn$ is Eulerian.

\begin{theorem}
\label{thm:gen_functions}
(a) $f$ is characterized uniquely by the following equations:
\begin{align}
\label{eq:RRT_gen_eq}
f = \frac{f^2}{1 - t f} + x, \quad \frac{\partial f}{\partial x}(0,0) = 1.
\end{align}
(b)
The collection $\bigl(F_T(t,x)\bigr)_{{r\geq1,}\atop{T \in K_r}}$ satisfies the following equations:
\begin{gather}
\label{eq:tree-pair_gen_eq}
F_\bigcdot
=
\frac{F_\bigcdot^2}{1-tF_\bigcdot} + x,
\\
F_T
=
\frac {F_T^2}{t^p-tF_T}
+
t^{p-1}\left(\frac{t^{p_1}}{t^{p_1}-tF_{T_1}}\cdots\frac{t^{p_k}}{t^{p_k}-tF_{T_k}} - 1\right),
\quad
T = C(T_1,\ldots,T_k),
\nonumber
\end{gather}
where in the latter equation, we have set $p \coloneqq d(T)$, $p_i \coloneqq d(T_i)$.
\end{theorem}

\subsection{Counting faces in $K_r$}

In this subsection we will prove part (a) of Thm.~\ref{thm:gen_functions}.
We begin by defining the operation of concatenation.
Set $T_r^\max$ to be the unique element of $K_r$ with $d(T_\max^r) = r-2$.
(That is, $T_\max^1$ is a single vertex, and for $r\geq2$, $T_\max^r$ is a root vertex with $r$ incoming neighbors.)

\begin{defprop}
\label{defprop:concat}
Fix $k\geq1$ and a $k$-tuple $(T_i \in K_{p_i,q_i})_{1\leq i\leq k}$ of stable RRTs.
Then the \emph{concatenation} of this data is the stable RRT $C(T_1,\ldots,T_r)$ defined by attaching $T_1,\ldots,T_r$ to the $r$ leaves of $T_\max^r$.
The concatenated tree has $\sum_{i=1}^k q_i$ leaves and the following dimension:
\begin{align}
d\bigl(C(T_1,\ldots,T_r)\bigr) = \begin{cases}
p_1, & k=1, \\
\sum_{i=1}^k p_i + k - 2, & k \geq 2.
\end{cases}
\end{align}
\end{defprop}

\begin{proof}
The $k=1$ case is trivial, so from now on we assume $k\geq 2$.

Set $T \coloneqq C(T_1,\ldots,T_r)$.
We must establish the formula $d(T) = \sum_{i=1}^k p_i + k - 2$ and show that $T$ has $\sum_{i=1}^k q_i$ leaves.
The latter equality follows immediately from the definition of $T$.
The former equality follows from a calculation:
\begin{align}
d(T) = \sum_{i=1}^k q_i - \#T_\inte - 1 = \sum_{i=1}^k q_i - \left(\sum_{i=1}^k \#(T_i)_\inte + 1\right) - 1 &= \sum_{i=1}^k \left(q_i - \#(T_i)_\inte - 1\right) + k - 2 \nonumber
\\
&= \sum_{i=1}^k p_i + k - 2.
\end{align}
\end{proof}

We are now ready to prove part (a) of Thm.~\ref{thm:gen_functions}.

\begin{proof}[Proof of Thm.~\ref{thm:gen_functions}(a)]
{\bf Step 1:}
{\it $f$ satisfies \eqref{eq:RRT_gen_eq}.}

\medskip

\noindent
For $m \geq 0, r \geq 2$, there is a bijective correspondence between the following two sets:
\begin{align}
\label{eq:RRT_decomp}
K_{m,r} \quad\longleftrightarrow\quad \bigsqcup_{k=2}^\infty \: \bigsqcup_{\substack{p_1+\cdots+p_k = m-k+2
\\
q_1+\cdots+q_k=r \\
p_i \geq 0, \: q_i \geq 1}} K_{p_1,q_1} \times \cdots \times K_{p_k,q_k}
\end{align}
Indeed, the right-to-left direction of this correspondence is defined by concatenation, as in Def.-Prop.~\ref{defprop:concat}, and the left-to-right direction follows from the observation that concatenation is reversible.

For $m \geq 0, r \geq 2$, the following recursion is an immediate consequence of \eqref{eq:RRT_decomp}:
\begin{gather}
a_{mr} = \sum_{k=2}^\infty \: \sum_{\substack{p_1+\cdots+p_k=m-k+2
\\
q_1+\cdots+q_k=r
\\
p_i \geq 0, \: q_i \geq 1}} a_{p_1 q_1}\cdots a_{p_k q_k}.
\end{gather}
Moreover, we have $a_{01} = 1$ and $a_{m1} = 0$ for $m \geq 1$.
These equations in $(a_{mr})$ imply the following eq uation in $f$:
\begin{align}
f = \sum_{k = 2}^\infty t^{k-2}f^k + x,
\end{align}
which can be rewritten as
\begin{align}
f = \frac{f^2}{1 - t f} + x.
\end{align}
The equation $a_{01}=1$ is equivalent to $\tfrac{\partial f}{\partial x}(0,0) = 1$, so $f$ satisfies \eqref{eq:RRT_gen_eq}.

\medskip

\noindent
{\bf Step 2:}
{\it \eqref{eq:RRT_gen_eq} has only one solution.}

\medskip

\noindent
Next, we show that \eqref{eq:RRT_gen_eq} has only one solution.
The first equation in \eqref{eq:RRT_gen_eq} has two solutions:
\begin{align}
f_\pm \coloneqq \frac{1+tx\pm\sqrt{1-4x-2tx+t^2x^2}}{2(1+t)}.
\end{align}	
Then $f_\pm$ has $\tfrac{\partial f_\pm}{\partial x}(0,0) = \mp1$, so $f_- = f$ is the unique solution of \eqref{eq:RRT_gen_eq}.
\end{proof}

\subsection{Counting faces in \texorpdfstring{$W_\bn$}{2Mn}}

As in the previous section, we will now characterize the family $\bigl(F_T(t,\bx)\bigr)_{{r\geq1,}\atop{T \in K_r}}$ in terms of equations that the functions in this family satisfy.
We begin with a concatenation operation on tree-pairs.
For any $r\geq 0$ and $\bn \in \bZ_{\geq0}^r\setminus\{\bzero\}$, let $2T_\bn^\max$ denote the unique element of $W_\bn$ with $d(2T_\max^\bn) = |\bn| + r - 3$.

\begin{defprop}
\label{defprop:2concat}
Fix $k\geq 1$, $\ba \in \bZ_{\geq0}^k\setminus\{\bzero\}$, a $k$-tuple of stable RRTs $(T_i \in K_{p_i,q_i})_i$, and for every $i \in [1,r]$ an $a_i$-tuple of stable tree-pairs $(2T_{ij} \in W_{T_i,P_i,\boldQ_{ij}})_j$.
Then the \emph{concatenation} of this data is the stable tree-pair $2C\bigl((T_i),(2T_{ij})\bigr) = \bigl(T_b \sr{\pi}{\to} T_s\bigr)$ defined in the $k=1, a_1=1$ case to be $2T_{11}$, and in the remaining cases like so:
\begin{itemize}
\item $T_s$ is the concatenation $C\bigl((T_i)\bigr)$ as in Def.-Prop.~\ref{defprop:concat}.

\item We define $T_b$ by attaching the root of $2T_{ij}$ to $\mu_{ij} \in V_\mk(2T_\ba^\max)$.
\end{itemize}
In the case $\ba \neq (1)$, $2C\bigl((T_i),(2T_{ij})\bigr)$ is a stable tree-pair of type $\bigl(\sum_{j=1}^{a_1} \boldQ_{1j},\ldots,\sum_{j=1}^{a_k} \boldQ_{kj}\bigr)$, and has the following dimension:
\begin{align}
\label{eq:dim_two-concat}
d\Bigl(2C\bigl((T_i),(2T_{ij})\bigr)\Bigr) = \sum_{i=1,j=1}^{k,a_i} P_{ij} - \sum_{i=1}^k (a_i-1)p_i + |\ba| + k - 3.
\end{align}
\end{defprop}

\begin{proof}
The only thing we need to check is the dimension formula \eqref{eq:dim_two-concat}.
\begin{align}
d\Bigl(2C\bigl((T_i),(T_{ij})\bigr)\Bigr)
&=
\sum_{i=1,j=1}^{k,a_i} |\boldQ_{ij}| + \sum_{i=1}^k q_i - \#V_\comp^1(T_b) - \#(T_s)_\inte - 2
\nonumber
\\
&= \sum_{i=1,j=1}^{k,a_i} |\boldQ_{ij}| + \sum_{i=1}^k q_i - \biggl(\,\sum_{i=1,j=1}^{k,a_i} \#V_\comp^1(T_{b,ij}) + \delta_{k1}\biggr) \\
&\hspace{2.8in} - \biggl(\sum_{i=1}^k \#(T_{s,i})_\inte + (1-\delta_{k1})\biggr) - 2
\nonumber
\\
&= \sum_{i=1,j=1}^{k,a_i} \Bigl(|\boldQ_{ij}| + q_i - \#V_\comp^1(T_{b,ij}) - \#(T_{s,i})_\inte - 2\Bigr)
\nonumber
\\
&\hspace{1.75in}
-
\sum_{i=1}^k (a_i-1)\bigl(q_i - \#(T_{s,i})_\inte - 1\bigr)
+ \sum_{i=1}^k a_i + k - 3
\nonumber
\\
&= \sum_{i=1,j=1}^{k,a_i} P_{ij} - \sum_{i=1}^k (a_i-1)p_i + |\ba| + k - 3.
\nonumber
\end{align}
\end{proof}

Now that we have defined concatenation for tree-pairs, we are ready to derive a collection of equations that $\bigl(F_T(t,\bx)\bigr)$ satisfies.

\begin{proof}[Proof of Thm.~\ref{thm:gen_functions}(b).]
By \cite[Lemma~3.9]{b:2ass}, the following sets are in bijection for $m \geq 0, n \geq 1$:
\begin{align}
\label{eq:dot_decomp}
W_{\bigcdot,m,n} \simeq K_{m,n}.
\end{align}
\noindent It follows from this sublemma that $A_{\bigcdot,m,n} = a_{m,n}$, hence that $F_\bigcdot = f$, where $f$ is the generating function defined in the last subsection.
That $F_\bigcdot$ satisfies the first equation in \eqref{eq:tree-pair_gen_eq} now follows from Thm.~\ref{thm:gen_functions}(a).

Next, fix $r \geq 2$ and $T \in K_r$, and write $T = C(T_1,\ldots,T_k)$.
The following sets are then in bijection:
\begin{align}
\label{eq:W_recursion}
W_{T,m,\bn}
&\simeq
\bigsqcup_{a\geq2} \:\: \prod_{{{1\leq j\leq a,}\atop{\sum m_j = m + (a-1)p - a + 2}}\atop{\bn = \sum \boldQ_j}} W_{T,m_j,\boldQ_j}
\\
&\hspace{1.5in}
+
\bigsqcup_{\ba \in \bZ_{\geq0}^k\setminus\{\bzero\}} \prod_{{{1\leq i\leq k,\:1\leq j\leq a_i,}
\atop
{(\sum \boldQ_{1j},\ldots,\sum \boldQ_{kj}) = \bn,}}
\atop
{\sum m_{ij} = m + \sum a_i(p_i-1) - \sum p_i - k + 3}} W_{T_i,m_{ij},\boldQ_{ij}},
\nonumber
\end{align}
where we have denoted $p \coloneqq d(T)$.
This follows from concatenation of tree-pairs, analogously with \eqref{eq:RRT_decomp}.
Indeed, either the root of $T_b$ has one solid incoming edges, which in turn has $\ell \geq 2$ dashed incoming edges; or it has $k \geq 2$ solid incoming edges, which in turn have $\ell_1,\ldots,\ell_k$ dashed incoming edges.
Note that in the latter case, from any one of the dashed edges attached to the $j$-th solid incoming edge of the root we can extract a picture $T_b' \sr{f'}{\to} T_j$, where $T_j$ is the $j$-th branch of $T$.
\eqref{eq:W_recursion} implies the following equality on generating functions:
\begin{align}
F_T
=
t^{-p+2}\sum_{a=2}^\infty
t^{a(-p+1)}F_T^a
+
t^{\sum p_i + k - 3}
\Biggl(
\prod_{i=1}^k
\sum_{a_i \geq 0}
t^{a(-p_i+1)}F_{T_i}^a
- 1
\Biggr)
,
\end{align}
which is equivalent to the equation
\begin{align}
F_T
=
t^{p-2}\frac{t^{-2p+2}F_T^2}{1-t^{-p+1}F_T}
+
t^{\sum_i p_i + k - 3}
\biggl(
\frac1{1-t^{-p_1+1}F_{T_1}}\cdots\frac1{1-t^{-p_k+1}F_{T_k}}
-
1
\biggr),
\end{align}
which simplifies to
\begin{align}
F_T
=
\frac {F_T^2}{t^p-tF_T}
+
t^{p-1}\left(\frac{t^{p_1}}{t^{p_1}-tF_{T_1}}\cdots\frac{t^{p_k}}{t^{p_k}-tF_{T_k}} - 1\right).
\end{align}
\end{proof}

\section{The 2-associahedra are Eulerian}
\label{s:eulerian}

In this section, we will use the functional equations we derived in Thm.~\ref{thm:gen_functions} to prove Thm.~\ref{thm:eulerian}.
We will split our proof into three parts:
\begin{enumerate}
\item
In Lemma~\ref{lem:Wn-hat_balanced}, \S\ref{ss:Wn-hat}, we use Thm.~\ref{thm:gen_functions} to show that $\wh{W_\bn}$ is balanced.

\item
In Lemma~\ref{lem:A_fib_prod_of_Wms}, \S\ref{ss:fib_prods}, we use Lemma~\ref{lem:Wn-hat_balanced} and the recursive structure of 2-associahedra to show that any reduced fiber product $W_{\bm_1}\wt\times_{K_r}\cdots\wt\times_{K_r} W_{\bm_k}$ is balanced.

\item
Finally, we show in Lemma~\ref{lem:remaining_intervals}, \S\ref{ss:remaining_intervals} that all remaining intervals (i.e.\ those intervals in $W_\bn$ with lower bound not equal to $F_\bn^\min$) are balanced.
\end{enumerate}

\noindent
Putting these results together, we now prove Thm.~\ref{thm:eulerian}.

\begin{proof}[Proof of Thm.~\ref{thm:eulerian}]
Fix $F_1, F_2, \in W_\bn$ with $F_1 < F_2$.
If $F_1 = F_\bn^\min$, then Lemma~\ref{lem:A_fib_prod_of_Wms} implies that $[F_1,F_2]$ is balanced.
Otherwise, it follows from Lemma~\ref{lem:remaining_intervals} that $[F_1,F_2]$ is balanced.
\end{proof}

\subsection{Definitions and basic facts about ranked posets}

\begin{definition}
A \emph{ranked poset} is a poset $P$ together with a rank function $d\colon P \to \bZ$, such that:
\begin{itemize}
\item
$x < y$ implies $d(x) < d(y)$.

\item
If $y$ covers $x$, then $d(y) = d(x)+1$.
\end{itemize}
For any elements $x < y$ in $P$, we define the \emph{alternating sum} of the interval $[a,b]$ to be
\begin{align}
A\bigl([x,y]\bigr)
\coloneqq
\sum_{z \in [x,y]}
(-1)^{d(z)}.
\end{align}
An interval $[x,y]$ is \emph{balanced} if we have $A\bigl([x,y]\bigr) = 0$, and $P$ is \emph{Eulerian} if, for every $x < y$, $[x,y]$ is balanced.
\null\hfill$\triangle$
\end{definition}

\begin{definition}[Reduced Product]
Given two ranked posets $P$ and $Q$, we define their \textit{reduced product} to be the poset
\begin{align}
P \:\wt\times\: Q
\coloneqq
((P_{\textrm{min}},P_{\textrm{max}}]\times (Q_{\textrm{min}},Q_{\textrm{max}}])\cup\{(P_{\textrm{min}}, Q_{\textrm{min}})\},
\end{align}
with partial order and rank induced by the product poset for the left side of the union, and rank $d(P_{\textrm{min}}) + d(Q_{\textrm{min}}) + 1$ for the newly-constructed minimal element.
Note that here we choose the convention that the induced rank for other elements be the sum of their ranks.
By iterating this operation, we can define the reduced product of an arbitrary number of ranked posets.
\null\hfill$\triangle$
\end{definition}

\begin{lemma}
\label{lem:redprod}
For any ranked posets $P_1,\ldots,P_k$, the identity $A\bigl(P_1 \:\wt\times\: \cdots \:\wt\times\: P_k\bigr) = A(P_1)\cdots A(P_k)$ holds.
\end{lemma}

\begin{proof}
{\bf Step 1:}
{\it We establish the $k=2$ case.}

\medskip

\noindent
This is a straightforward computation:
\begin{align}
A\bigl(\bigl(P_1 \wt\times P_2\bigr)\setminus\{(P_1^\min,P_2^\min)\}\bigr)
&=
\sum_{F\in (P_1\setminus \{P_1^\min\})\times (P_2\setminus \{P_2^\min\})}(-1)^{d(F)}
\\
&=
\sum_{(F_1, F_2) \in ((P\setminus \{P_\min\}), (Q\setminus \{Q_\min\}))}(-1)^{d(F_1)+d(F_2)}
\nonumber
\\
& =\sum_{F_1\in (P\setminus\{P_1^\min\})} (-1)^{d(F_1)} \cdot \sum_{F_2\in (P_2\setminus\{P_2^\min\})} (-1)^{d(F_2)}
\nonumber
\\
&=
d(P_1^\min)d(P_2^\min)
=
-d\bigl((P_1^\min,P_2^\min)\bigr).
\nonumber
\end{align}

\medskip

\noindent
{\bf Step 2:}
{\it We prove the lemma.}

\medskip

\noindent
Induction, using Step 1.
\end{proof}

\subsection{$W_\bn$ is balanced}
\label{ss:Wn-hat}

In this subsection, we will prove that $\wh{W_\bn}$ is balanced.
We will derive this as a consequence of Thm.~\ref{thm:gen_functions}.

\begin{lemma}
\label{lem:FT_at_-1}
For every $r \geq 1$ and $T \in K_r$, the generating function $F_T$ satisfies the following identity when we evaluate at $t=-1$:
\begin{align}
F_T(-1,\mathbf{x}) = (-1)^p\biggl(\frac{1}{\prod_{i=1}^r(1-x_i)} -1 \biggl),
\end{align}
where we set $p \coloneqq d(T)$.
\end{lemma}

\begin{proof}
{\bf Step 1:}
{\it We establish the $r=1$ case.}

\medskip

\noindent
By the $T=\bullet$ case of \eqref{eq:tree-pair_gen_eq}, $F_\bullet$ satisfies
\begin{align}
F_\bullet(-1,x) = \frac{F_\bullet(-1,x)^2}{1+F_\bullet(-1,x)} + x.
\end{align}
This yields
\begin{align}
\label{eq:F_bullet_-1_formula}
F_\bullet(-1,x)
=
\frac1{1-x}-1.
\end{align}

\medskip

\noindent
{\bf Step 2:}
{\it We prove the general case by strong induction on $r$.}

\medskip

\noindent
Choose $T_i \in K_{p_i}$, $1 \leq i \leq k$, such that $T = C(T_1,\ldots,T_k)$.
For convenience, we define and simplify a quantity $X$:
\begin{align}
X
=
\prod_{i=1}^k \frac{(-1)^{p_i}}{(-1)^{p_i} + F_{T_i}}
- 1
=
\prod_{i=1}^k \frac{(-1)^{p_i}}{(-1)^{p_i} + \Bigl(\frac{(-1)^{p_i}}{\prod_{j = 1}^{\ell_i}(1-y_{ij})} + (-1)^{p_i - 1}\Bigr)} - 1
&=
\prod_{i=1}^k\prod_{j=1}^{\ell_i}(1-y_{ij}) - 1
\nonumber
\\
&=
\prod_{i=1}^r(1-x_i) - 1,
\end{align}
where in the second equality we have used the inductive hypothesis.
This allows us to deduce the inductive step of the current lemma from \eqref{eq:tree-pair_gen_eq}.
Indeed, clearing denominators in that equation yields
\begin{align}
(-1)^pF_T
=
-X + F_T(-1)^{p-1}X,
\end{align}
and solving this equation for $F_T$ yields
\begin{align}
F_T
=
\frac{-X}{(-1)^p(1 + X)}
=
\frac{(-1)^p + (-1)^{p-1}\prod_{i = 1}^r(1-x_i)}{\prod_{i=1}^r (1-x_i)}
=
(-1)^p\big(\frac{1}{\prod_{i=1}^r(1-x_i)} -1 \big).
\end{align}
\end{proof}

\begin{lemma}
\label{lem:Wn-hat_balanced}
For any $r \geq 1$ and $\bn \in \bZ_{\geq0}^r\setminus\{\bzero\}$, $\wh{W_\bn}$ is balanced.
\end{lemma}

\begin{proof}
{\bf Step 1:}
{\it We show that $\wh{K_r}$ is balanced.}

\medskip

\noindent
By an argument identical to the derivation of \eqref{eq:F_bullet_-1_formula}, the generating function $f$ associated to $(K_r)$ satisfies the following identity:
\begin{align}
f(-1,x)
=
\frac1{1-x}-1
=
1 + x + x^2 + \cdots.
\end{align}
We have $A\Bigl(\wh{K_r}\Bigr) = [f(-1,x)]_{x^r} - 1$, where the first term on the right-hand side indicates the coefficient on $x^r$ in $f(-1,x)$ and where the second term comes from the minimal element in $\wh{K_r}$, so the last displayed equation implies $A\Bigl(K_r\Bigr) = 0$.

\medskip

\noindent
{\bf Step 2:}
{\it We prove the theorem.}

\medskip

\noindent
We interpret $A\Bigl(\wh{W_\bn}\Bigr)$ in terms of the generating functions $F_T$, then use Step 1 and Lemma~\ref{lem:FT_at_-1} to prove the theorem:
\begin{align}
A\Bigl(\wh{W_\bn}\Bigr)
=
\sum_{T \in K_r}
\bigl[F_T(-1,\bx)\bigr]_{\bx^\bn} - 1
\:&\sr{\text{Lem.} \ref{lem:FT_at_-1}}{=}\:
\sum_{T \in K_r}
\biggl[(-1)^{d(T)}\biggl(\frac1{\prod_{i=1}^r(1-x_i)}-1\biggr)\biggr]_{\bx^n} - 1
\\
&\hspace{0.145in}=
A(K_r)\left[\frac1{\prod_{i=1}^r (1-x_i)-1}-1\right]_{\bx^n} - 1
\nonumber
\\
&\hspace{0.145in}=
\bigl[(1+x_1+x_1^2+\cdots)\cdots(1+x_r+x_r^2+\cdots) - 1\bigr]_{\bx^n} - 1
\nonumber
\\
&\hspace{0.145in}=
0.
\nonumber
\end{align}
\end{proof}

\subsection{Intervals of the form $[F_\bn^\min, 2T]$ are balanced}
\label{ss:fib_prods}

In this subsection, we will generalize Lemma~\ref{lem:Wn-hat_balanced} to the statement that for every $2T \in W_\bn$, the sublevel set $[F_\bn^\min,2T]$ is balanced.
The key is the recursive structure of the 2-associahedra, which implies that $[F_\bn^\min, 2T]$ decomposes as a completed fiber product of products of 2-associahedra.

\begin{lemma}
\label{lem:A_fib_prod_of_Wms}
Fix $r \geq 0$ and $\bm_1,\ldots,\bm_k \in \bZ_{\geq0}^r\setminus\{\bzero\}$.
Then the following identity holds, where the expression in the alternating sum denotes a fiber product with respect to the forgetful map $\pi\colon W_{\bm_i} \to K_r$:
\begin{align}
A\biggl(\overset{K_r}{\prod}_{1\leq i\leq k}
W_{\bm_i}\biggr)
=
1.
\end{align}
\end{lemma}

\begin{proof}
{\bf Step 1:}
{\it We compute $A\bigl(\pi_{\bm_1}^{-1}(T) \times \cdots \times \pi_{\bm_k}^{-1}(T)\bigr)$, where $T$ is any element of $K_r$ and $\pi_{\bm_i} \colon W_{\bm_i} \to K_r$ is the forgetful map.}

\medskip

\noindent
\begin{align}
A\bigl(\pi_{\bm_1}^{-1}(T) \times \cdots \times \pi_{\bm_k}^{-1}(T)\bigr)
=
\sum_{{F_i \in W_{\bm_i},}
\atop
{1\leq i\leq k}}
(-1)^{d(F_1,\ldots,F_k)}
&=
\sum_{{F_i \in W_{\bm_i},}
\atop
{1\leq i\leq k}}
(-1)^{d(F_1)+ \cdots + d(F_k) - (k-1)d(T)}
\nonumber
\\
&=
(-1)^{-(k-1)d(T)}
\prod_{i=1}^k
\sum_{F_i\in \pi^{-1}_{\bm_i}(T)}
(-1)^{d(F_i)}
\nonumber
\\
&=
(-1)^{-(k-1)d(T)}
\prod_{i=1}^k
(-1)^{d(T)}A(W_{\bm_i})
\nonumber
\\
&=
(-1)^{d(T)}.
\nonumber
\end{align}

\medskip

\noindent
{\bf Step 2:}
{\it We prove the lemma.}

\noindent
\begin{align}
A\biggl(\overset{K_r}{\prod}_{1\leq i\leq k}
W_{\bm_i}\biggr)
=
A\biggl(\bigsqcup_{T \in K_r}
\bigl(\pi_{\bm_1}^{-1}(T)\times\cdots\times\pi_{\bm_r}^{-1}(T)\bigr)\biggr)
&=
\sum_{T \in K_r}
A\bigl(\pi_{\bm_1}^{-1}(T) \times \cdots \times \pi_{\bm_k}^{-1}(T)\bigr)
\nonumber
\\
&=
\sum_{T \in K_r} (-1)^{d(T)}
\\
&=
1.
\nonumber
\end{align}
\end{proof}

\begin{lemma}
\label{lem:F_min_2T_balanced}
Fix $r\geq1$, $\bn \in \bZ_{\geq0}^r\setminus\{\bzero\}$, and $2T \in W_\bn$.
Then the sublevel set $[F_\bn^\min,2T]$ is balanced.
\end{lemma}

\begin{proof}
By the \textsc{(recursive)} part of \cite[Thm.~4.1]{b:2ass}, the following posets are isomorphic:
\begin{align}
\label{eq:fiber_product_decomp}
\prod_{
{\alpha \in V_\comp^1(T_b),}
\atop
{\incom(\alpha)=(\beta)}
}
W_{\#\!\incom(\beta)}^\tree
\times
\prod_{\rho \in V_\inte(T_s)} \prod^{K_{\#\!\incom(\rho)}}_{
{\alpha\in V_\comp^{\geq2}(T_b)\cap \pi^{-1}\{\rho\},}
\atop
{\incom(\alpha)=(\beta_1,\ldots,\beta_{\#\!\incom(\rho)})}
}
\hspace{-0.25in} W_{\#\!\incom(\beta_1),\ldots,\#\!\incom(\beta_{\#\!\incom(\alpha)})}
\simeq
(F_\bn^\min,2T]
\subset
W_\bn
\end{align}
It now follows from Lemmas~\ref{lem:redprod} and \ref{lem:A_fib_prod_of_Wms} that $A\bigl([F_\bn^\min,2T]\bigr) = 0$.
\end{proof}

\subsection{All remaining intervals are balanced}
\label{ss:remaining_intervals}

\begin{definition}
Fix a 2-bracketing $(\sB,\stB) \in W_\bn^\br$.
A bracket $B \in \sB$ is \emph{removable} if it is not a singleton and does not contain all elements of $(1,\ldots,r)$.
A 2-bracket $\btB  \in \stB$ is \emph{removable} if it is not a singleton and is not the maximal 2-bracket $\btB_\bn^\max \coloneqq \{(i,j) \:|\: 1\leq i\leq r, 1\leq j\leq n_i\} \subset \stB$.
\end{definition}

\begin{lemma}
\label{lem:remaining_intervals_superlevel_set}
Fix $r \geq 1$, $\bn \in \bZ_{\geq0}^r\setminus\{\bzero\}$, and $\bigl(\sB, \stB\bigr) \in W_\bn \setminus \{F_\bn^\max\}$.
Then the interval
\begin{align}
I
=
\bigl[\bigl(\sB,\stB\bigr), F_\bn^\max\bigr]
\subset
W_\bn^\br
\end{align}
is balanced.
\end{lemma}

\begin{proof}
We will induct on $r$, and subinduct on the number of 2-brackets in $\stB$.
If $r = 1$, this lemma follows from the identification $W_n \simeq K_n$ made in \cite[Lemma~3.9]{b:2ass}.
In the remainder of the proof, we fix $\bigl(\sB,\stB\bigr) \in W_\bn^\br$ for $r \geq 2$, $\bn \in \bZ_{\geq0}^r$, and use the inductive hypothesis to show that $I$ is balanced.

\medskip

\noindent
{\bf Step 1:}
{\it We prove that $I$ is balanced in the case that $\sB$ contains no removable brackets and every removable 2-bracket $\btB \in \stB$ has $\pi(\btB) = (1,\ldots,r)$.}

\medskip

\noindent
We establish this step by strong induction on $d\bigl(F_\bn^\max\bigr) - d\bigl(\sB,\stB\bigr)$.
This step is trivial when this difference in dimension is $1$.

Next, consider the case when this difference is at least $2$.
There must then be a 2-bracket $\btB \in \stB$ which properly contains another 2-bracket in $\stB$, and which is maximal in $\stB \setminus \{\btB_\bn^\max\}$.
Indeed, if there were no such 2-bracket, then every 2-bracket in $\stB \setminus \{\btB_\bn^\max\}$ would be minimal, so by \cite[Def.~3.12, \textsc{(marked seams are unfused)}]{b:2ass}, we would have $d\bigl(\sB,\stB\bigr) = d\bigl(F_\bn^\max\bigr) - 1$.

Denote by $\stB'$ the 2-brackets that $\btB$ contains, and by $\stB''$ the 2-brackets other than $\btB_\bn^\max$ that $\btB$ does not include.
Then $\bigl(\sB,\stB'\bigr) \in W_{\bn'}^\br$, $\bigl(\sB,\stB''\bigr) \in W_{\bn''}^\br$ are legal 2-bracketings.
Using this notation, we can decompose decompose $I$ according to whether an element of $I$ does or does not contain $\btB$:
\begin{align}
\label{eq:decomp_in_superlevel_set_only_wide}
I
&=
\bigl[\bigl(\sB,\stB\bigr), F_\bn^\max\bigr]
\\
&\simeq
\Bigl(
\bigl[\bigl(\sB,\stB'\bigr),F_{\bn'}^\max\bigr]
\times
\bigl(\bigl[\bigl(\sB,\stB''\bigr),F_{\bn''}^\max\bigr]\setminus\{F_{\bn''}^\max\}\bigr)
\Bigr)
\sqcup
\bigl[\bigl(\sB,\stB\setminus\{\btB\}\bigr), F_\bn^\max\bigr].
\nonumber
\end{align}
Indeed, given an element of the product that appears on the left-hand side of the final expression, we can insert the 2-bracketing in $W_{\bn'}^\br$ into $\btB$ and use the 2-bracketing in $W_{\bn''}^\br$ to bracket the remaining elements.
We do not allow the 2-bracketing in $W_{\bn''}^\br$ to be the top element, because the presence of $\btB$ would then violate the \textsc{(marked seams are unfused)} property.
On the other hand, an element in the right-hand part of the final expression in \eqref{eq:decomp_in_superlevel_set_only_wide} is simply an element of $I$ that does not include $\btB$.

It follows from \eqref{eq:decomp_in_superlevel_set_only_wide} and the inductive hypothesis that $I$ is balanced.

\medskip

\noindent
{\bf Step 2:}
{\it We prove that $I$ is balanced in the case that $\stB$ does not contains a 2-bracket of the form $\btB = \bigl(B,(2B_i)\bigr) \neq \btB_\bn^\max$ with $2B_i = \{(i,j) \:|\: 1 \leq j \leq n_i\}$ for all $i$.}

\medskip

\noindent
Fix $B \in \sB$, and set $s \coloneqq \#\!B$ and $\bn(B) \coloneqq (n_i)_{i \in B} \in \bZ_{\geq0}^s$.
We can then define a 2-bracketing $\bigl(\sB(B), \stB(B)\bigr)$ like so:
\begin{align}
\label{eq:decomp_in_superlevel_set_by_bracket}
\sB(B)
\coloneqq
\{(i)\}_{i \in B}
\cup
\{B\},
\qquad
\stB(B)
\coloneqq
\{(i,j)\}_{{i \in B,}
\atop
{1 \leq j \leq n_i}}
\cup
\bigl(\stB \cap \pi^{-1}\{B\}\bigr)
\cup
\{\btB_{\bn(B)}^\max\}.
\end{align}
That is, $\bigl(\sB(B),\stB(B)\bigr)$ contains only those nonremovable elements of $\stB$ that lie over $B$.
Using this construction, we can decompose $I$ in terms of the brackets in $\sB$:
\begin{align}
I
\simeq
\prod_{B \in \sB}
\Bigl[\bigl(\sB(B),\stB(B)\bigr),F_{\bn(B)}^\max\Bigr].
\end{align}
Indeed, given an element of this product, we can take the union of all the brackets and 2-brackets to obtain an element of $I$, and this correspondence is clearly bijective.

It follows from \eqref{eq:decomp_in_superlevel_set_by_bracket} and Step 1 that $I$ is balanced.

\medskip

\noindent
{\bf Step 3:}
{\it We prove the lemma.}

\medskip

\noindent
If $I$ satisfies the hypothesis of Step 2, we are done.
Otherwise, fix a 2-bracket $\btB \in \stB$ of the sort excluded in Step 2.
If $\btB$ does not properly contain any element of $\stB \cap \pi^{-1}\{B\}$, then we have
\begin{align}
I
\simeq
\bigl[\bigl(\sB,\stB\setminus\{\btB\}\bigr),
F_\bn^\max\bigr],
\end{align}
so $I$ is balanced by the inductive hypothesis.

Next, suppose that $\btB$ does properly contain an element of $\stB \cap \pi^{-1}\{B\}$.
Similarly to \eqref{eq:decomp_in_superlevel_set_only_wide}, we can decompose $I$ according to whether an element of $I$ does or does not contain $\btB$:
\begin{align}
\label{eq:decomp_in_superlevel_set_take_off_big_bubble}
I
\simeq
\bigl[\bigl(\sB,\stB\setminus\{\btB\}\bigr),F_\bn^\max\bigr]
\sqcup
\bigl[\bigl(\sB,\stB\setminus\{\btB\}\bigr),F_\bn^\max\bigr].
\end{align}
Indeed, if we take an element of the first copy of $\bigl[\bigl(\sB,\stB\setminus\{\btB\}\bigr),F_\bn^\max\bigr]$, we can add $\btB$ to produce a 2-bracketing in $I$ that includes $\btB$.
An element of the second copy of $\bigl[\bigl(\sB,\stB\setminus\{\btB\}\bigr),F_\bn^\max\bigr]$ corresponds to a 2-bracketing in $I$ that does not include $\btB$.

It follows from \eqref{eq:decomp_in_superlevel_set_take_off_big_bubble} and the inductive hypothesis that $I$ is balanced.
\end{proof}

\begin{lemma}
\label{lem:remaining_intervals}
Fix $r \geq 1$, $\bn \in \bZ_{\geq 0}^r\setminus \{\bzero\}$, and $\bigl(\sB^{(i)}, \stB^{(i)}\bigr) \in W_{\bn}^\br$ for $i \in \{1,2\}$ with
\begin{align}
\bigl(\sB^{(2)}, \stB^{(2)}\bigr) < \bigl(\sB^{(1)}, \stB^{(1)}\bigr).
\end{align}
Then the interval $I = \bigl[\bigl(\sB^{(2)}, \stB^{(2)}\bigr), \bigl(\sB^{(1)}, \stB^{(1)}\bigr)\bigr]$ is balanced.
\end{lemma}

\begin{proof}
By \eqref{eq:fiber_product_decomp}, it suffices to show that in any fiber product $W_{\bm_1}\times_{K_r}\cdots\times_{K_r}W_{\bm_k}$, any nontrivial interval with upper bound $\bigl(F_{\bm_1}^\max,\ldots,F_{\bm_k}^\max\bigr)$ is balanced.
This follows from an argument very similar to our proof of Lemma~\ref{lem:remaining_intervals_superlevel_set}.
\end{proof}


\begin{thebibliography}{10}

\bibitem[Bo1]{b:2ass}
N.~Bottman.
\newblock 2-associahedra.
\newblock {\it Algebraic \& Geometric Topology} 19 (2019), no.\ 2, pp.\ 743--806.

\bibitem[Bo2]{b:realization}
N.~Bottman.
\newblock Moduli spaces of witch curves topologically realize the 2-associahedra.
\newblock Accepted (2018), {\it Journal of Symplectic Geometry}.

\bibitem[BoCa]{bc:Ainfty}
N.~Bottman, S.~Carmeli.
\newblock $(A_\infty,2)$-algebras and relative 2-operads.
\newblock Submitted to {Higher Structures}; available at \url{https://arxiv.org/abs/1811.05442}.

\bibitem[BW]{bw:compactness}
N.~Bottman, K.~Wehrheim.
\newblock Gromov compactness for squiggly strip shrinking in pseudoholomorphic quilts.
\newblock {\it Selecta Math. (N.S.)} 24 (2018), no.\ 4, pp.\ 3381--3443.

\bibitem[MWW]{mww}
S.~Ma'u, K.~Wehrheim, C.~Woodward.
\newblock $A_\infty$ functors for Lagrangian correspondences.
\newblock {\it Selecta Math.\ (N.S.)} 24 (2018), no.\ 3, 1913--2002.

\bibitem[St]{stanley:cd}
R.~Stanley.
\newblock Flag $f$-vectors and the $cd$-index.
\newblock {\it Mathematische\ Zeitschrift} 216 (1994), no.\ 3, pp.\ 483--499.

\end{thebibliography}
\end{document}